\documentclass[a4paper,12pt]{article}

\usepackage[usenames,dvipsnames]{pstricks}

\usepackage{amscd,amssymb,amsmath,graphicx,verbatim,hyperref}
\usepackage[OT1,T1]{fontenc}

\usepackage[all, cmtip]{xy}
\usepackage{pdfpages}

\usepackage{amsmath,amsthm,amsfonts,mathabx,tikz,tikz-cd,stmaryrd,url,mathtools, mathrsfs, tensor, pigpen}
\usepackage{hyperref}
\hypersetup{
	colorlinks,
	citecolor=black,
	filecolor=black,
	linkcolor=black,
	urlcolor=black
}

\usetikzlibrary{matrix,arrows,decorations.pathmorphing}

\newcounter{Definitioncount}
\newtheorem{theorem}{Theorem}
\newtheorem{lemma}[theorem]{Lemma}
\newtheorem{proposition}[theorem]{Proposition}

\theoremstyle{definition}

\newtheorem{remark}[theorem]{Remark}
\newtheorem{example}{Example}

\newtheoremstyle{fact}{\bigskipamount}{\medskipamount}{\upshape}{}{\itshape}{. }{ }{Fact}
\theoremstyle{fact}

\newtheoremstyle{quest}{\bigskipamount}{\medskipamount}{\upshape}{}{\itshape}{. }{ }{Question}
\theoremstyle{quest}

\newtheoremstyle{step}{2\bigskipamount}{\medskipamount}{\upshape}{}{\itshape}{. }{ }{\underline{Step~\thestep}}
\theoremstyle{step}

\renewcommand{\thestep}{\arabic{step}}

\newcommand{\po}{\ar@{}[dr]|{\text{\pigpenfont R}}}
\newcommand{\pb}{\ar@{}[dr]|{\text{\pigpenfont J}}}
\newcommand{\fpb}{\ar@{}[dr]|{\text{fakepb}}}

\newcommand{\ldual}[1]{\mathord{{\let\nolimits\relax\sideset{^\wedge}{}{#1}}}}
\newcommand{\laction}[2]{\mathord{{\let\nolimits\relax\sideset{^{#1}}{}{#2}}}}
\newcommand{\conj}[2]{\mathord{{\let\nolimits\relax\sideset{^{#1}}{}{#2}}}}

\newcommand{\xra}{\xrightarrow}
\newcommand{\xla}{\xleftarrow}
\newcommand{\era}{\twoheadrightarrow}
\newcommand{\ela}{\twoheadleftarrow}
\newcommand{\mra}{\rightarrowtail}

\def\CC{{\mathscr C}}

\def\CE{{\mathscr E}}

\def\CL{{\mathscr L}}
\def\CM{{\mathscr M}}

\def\CS{{\mathscr S}}

\def\CU{{\mathscr U}}

\makeatletter
\newcommand*\bigcdot{\mathpalette\bigcdot@{.5}}
\newcommand*\bigcdot@[2]{\mathbin{\vcenter{\hbox{\scalebox{#2}{$\m@th#1\bullet$}}}}}
\makeatother

\newcommand{\twocong}[2][0.5]{\ar@{}[#2] \save ?(#1)*{\cong}\restore}
\newcommand{\twoeq}[2][0.5]{\ar@{}[#2] \save ?(#1)*{=}\restore}
\newcommand{\ltwocell}[3][0.5]{\ar@{}[#2] \ar@{=>}?(#1)+/r 0.2cm/;?(#1)+/l 0.2cm/^{#3}}
\newcommand{\rtwocell}[3][0.5]{\ar@{}[#2] \ar@{=>}?(#1)+/l 0.2cm/;?(#1)+/r 0.2cm/^{#3}}
\newcommand{\utwocell}[3][0.5]{\ar@{}[#2] \ar@{=>}?(#1)+/d  0.2cm/;?(#1)+/u 0.2cm/_{#3}}
\newcommand{\dtwocell}[3][0.5]{\ar@{}[#2] \ar@{=>}?(#1)+/u  0.2cm/;?(#1)+/d 0.2cm/^{#3}}
\newcommand{\ultwocell}[3][0.5]{\ar@{}[#2] \ar@{=>}?(#1)+/dr  0.2cm/;?(#1)+/ul 0.2cm/^{#3}}
\newcommand{\urtwocell}[3][0.5]{\ar@{}[#2] \ar@{=>}?(#1)+/dl  0.2cm/;?(#1)+/ur 0.2cm/^{#3}}
\newcommand{\dltwocell}[3][0.5]{\ar@{}[#2] \ar@{=>}?(#1)+/ur  0.2cm/;?(#1)+/dl 0.2cm/^{#3}}
\newcommand{\drtwocell}[3][0.5]{\ar@{}[#2] \ar@{=>}?(#1)+/ul  0.2cm/;?(#1)+/dr 0.2cm/^{#3}}

\begin{document}
\author{Ross Street \footnote{The author gratefully acknowledges the support of Australian Research Council Discovery Grant DP160101519.} \\ 
\small{Centre of Australian Category Theory} \\
\small{Macquarie University, NSW 2109 Australia} \\
\small{<ross.street@mq.edu.au>}
}
\title{Span composition using fake pullbacks}
\date{\small{\today}}
\maketitle

\noindent {\small{\emph{2010 Mathematics Subject Classification:} 18B10, 18D05}
\\
{\small{\emph{Key words and phrases:} span; partial map; factorization system.}}

\begin{abstract}
\noindent The construction of a category of spans can be made 
in some categories $\CC$ which do not have pullbacks in the traditional
sense. The PROP for monoids is a good example of such a $\CC$.
The 2012 book concerning homological algebra by Marco Grandis
gives the proof of associativity of relations in a Puppe-exact category
based on a 1967 paper of M.\v{S}. Calenko. 
The proof here is a restructuring of that proof in the spirit of the first
sentence of this Abstract.
We observe that these relations are spans of EM-spans and that 
EM-spans admit fake pullbacks so that spans of EM-spans compose.  
Our setting is more general than Puppe-exact categories. 
\end{abstract}

\tableofcontents

\section*{Introduction}

The construction of a category of spans can be made 
in some categories $\CC$ not having pullbacks in the traditional
sense, only having some form of {\em fake pullback}. 
The PROP for monoids is a good example of such a $\CC$;
it has a forgetful functor to the category of finite sets which
takes fake pullbacks to genuine pullbacks.

As discussed in the book \cite{Grandis} by Marco Grandis,
relations in a Puppe-exact category $\CC$ are zig-zag diagrams of 
monomorphisms
and epimorphisms, not just jointly monomorphic spans as for a regular category
(see \cite{26} for example). Associativity of these zig-zag relations was proved
by M.\v{S}. Calenko \cite{Tsalenko1967} over 50 years ago; also see \cite{BP1969}
Appendix A.5, pages 140--142.   

The present paper is a restructuring of the associativity proof in the spirit of fake pullbacks. 
The original category $\CC$ does not even need to be pointed, 
but it should have a {\em suitable factorization system} $(\CE,\CM)$.
The fake pullbacks are constructed in what we call $\mathrm{Spn}(\CE,\CM)$,
not in $\CC$ itself, and there is no forgetful functor turning them into genuine pullbacks. 
The relations are spans in $\mathrm{Spn}(\CE,\CM)$. 
The main point in proving associativity of the span composition is that fake
pullbacks stack properly.

\section{Suitable factorization systems}

Let $(\CE,\CM)$ be a factorization system in the sense of \cite{FreydKelly} on a category $\CC$. That is, $\CE$ and $\CM$ are sets of morphisms of $\CC$ which contain the isomorphisms, are closed under composition, and satisfy the conditions:
\begin{itemize}
\item[FS1.] if $mu=ve$ with $e\in \CE$ and $m\in \CM$ then there exists a unique $w$
with $we=u$ and $mw=v$; 
\item[FS2.] every morphism $f$ factorizes $f=m\circ e$ with $e\in \CE$ and $m\in \CM$. 
\end{itemize}
If we write $f : A\era B$, we mean $f\in \CE$.
If we write $f : A\mra B$, we mean $f\in \CM$. 
Another way to express FS1 is to ask, for all $X\xra{e}Y\in \CE$ and $A\xra{m}B\in \CM$, that the square \eqref{factsystpb} should be a pullback.
\begin{equation}\label{factsystpb}
\begin{aligned}
\xymatrix{
\CC(Y,A) \ar[rr]^-{\CC(e,A)} \ar[d]_-{\CC(Y,m)} && \CC(X,A) \ar[d]^-{\CC(X,m)} \\
\CC(Y,B) \ar[rr]_-{\CC(e,A)} && \CC(X,B)}
\end{aligned}
\end{equation}
\begin{remark} If we were dealing with a factorization system on a bicategory $\CC$, we would ask \eqref{factsystpb} (with the associativity constraint providing a
natural isomorphism in the square) to be a bipullback. 
Also, in FS2, we would only ask $f\cong m\circ e$.  This is relevant to
Proposition~\ref{2wayFS} and Section~\ref{abstraction} below.  
\end{remark}

The factorization system $(\CE,\CM)$ is {\em suitable} when it satisfies:
 \begin{itemize}
\item[SFS1.] pullbacks of arbitrary morphisms along members of $\CM$ exist; 
\item[SFS2.] pushouts of arbitrary morphisms along members of $\CE$ exist; 
\item[SFS3.] the pullback of an $\CE$ along an $\CM$ is in $\CE$;
\item[SFS4.] the pushout of an $\CM$ along an $\CE$ is in $\CM$;
\item[SFS5.] a commutative square of the form
\begin{equation*}
\begin{aligned}
\xymatrix{
D \ar@{>->}[rr]^-{n} \ar@{->>}[d]_-{e} && B \ar@{->>}[d]^-{d} \\
A \ar@{>->}[rr]_-{m} && C }
\end{aligned}
\end{equation*}
is a pullback if and only if it is a pushout.
\end{itemize}

\begin{proposition}\label{jointly}
\begin{itemize}
\item[(i)] Spans of the form $X\ela S \mra Y$ are jointly monomorphic.
\item[(ii)] Cospans of the form $X\mra C \ela Y$ are jointly epimorphic.
\end{itemize}
\end{proposition}
\begin{proof}
A pullback of $X\mra C \ela Y$ exists by SFS1 and $X\mra C \ela Y$ is the
pushout of the resultant span by SFS5. Pushout cospans are jointly epimorphic.
This proves (ii), and (i) is dual.
\end{proof}
\begin{example}
\begin{itemize}
\item[1.] Take $\CC$ to be the category $\mathrm{Grp}$ of groups, $\CE$ to be the set of surjective morphisms and $\CM$ to be the set of injective morphisms. 
\item[2.] Take $\CC$ to be any Puppe-exact category as studied by Grandis \cite{Grandis},
$\CE$ the epimorphisms and $\CM$ the monomorphisms.
This includes all abelian categories.
\item[3.] Take $\CC$ to be the category $\mathrm{Spn}[\mathrm{Set}_{\mathrm{inj}}]$ of spans in the category of sets and injective functions, $\CE$ the $i^*$ and $\CM$ the $i_*$.
\item[4.] Take $\CC$ to be any groupoid with $\CE = \CM$ containing all morphisms. 
\end{itemize}
\end{example}

Now we remind the reader of Lemma 2.5.9 from \cite{Grandis}.
\begin{lemma}\label{MfactmorphLemma}
In a commutative diagram of the form
\begin{eqnarray}\label{Mfactmorph}
\begin{aligned}
\xymatrix{
A \ar@{->>}[rr]^-{d} \ar@{>->}[d]_-{\ell} && B \ar@{>->}[rr]^-{i} \ar@{>->}[d]^-{m} && C \ar@{>->}[d]^-{n} \\
X \ar@{->>}[rr]_-{e} && Y \ar@{>->}[rr]_-{j}  && Z \ ,
} 
\end{aligned}
\end{eqnarray}
the horizontally pasted square is a pullback if and only if both the component squares are pullbacks.
\end{lemma}
\begin{proof} ``If'' is true without any condition on the morphisms.
For the converse, using SFS1, take the pullback of $j$ and $n$ to obtain another
pastable pair of squares with the same left, right and bottom sides.
The top composites are equal. By factorization system properties and SFS3,
the new top is also a factorization of $i\circ d$ and thus isomorphic to the given factorization. So both of the old squares are also pullbacks.  
\end{proof}

We might call the diagram \eqref{Mfactmorph} an $\CM$-morphism of factorizations.
The dual of the lemma concerns pushouts in $\CE$-morphisms of factorizations;
it also holds since we did not use SFS5 in proving the Lemma. 
However condition SFS5 does tell us that
the left square of \eqref{Mfactmorph} is also a pushout when the pasted 
diagram is a pullback. 

\section{The bicategory of EM-spans}\label{BicatEMSpn}

Some terminology used here, for bicategories, spans and discrete fibrations, is explained in \cite{134}.

Let $(\CE,\CM)$ be a suitable factorization system on the category $\CC$.

We define a bicategory $\mathrm{Spn}(\CE,\CM)$ with the same objects as
$\CC$. The morphisms $(d,R,m) : U \to W$ are spans $U\ela R \mra W$ in $\CC$.
The 2-cells are the usual morphisms of spans. 
Composition is the usual composition of spans; this uses conditions SFS1, SFS3
and closure of $\CE$ under composition. 

Each $(X\xra{m}Y) \in \CM$ gives a morphism $m_* : X\xra{(1_X,X,m)} Y$ in
$\mathrm{Spn}(\CE,\CM)$ and each $(X\xra{e}Y) \in \CE$ gives a morphism 
$e^* : Y\xra{(e,X,1_X)} X$ in $\mathrm{Spn}(\CE,\CM)$.
Write $\CM_*$ for the class of all morphisms isomorphic to $m_*$ for some $m\in \CM$ and write
$\CE^*$ for the class of all morphisms isomorphic to $e^*$ with $e\in \CE$. 

Notice that 2-cells between members of $\CM_*$, 2-cells between members of $\CE^*$, and 2-cells from a member of $\CE^*$ to a member of $\CM_*$, are all invertible.  
 
Proposition~\ref{jointly} tells us that the bicategory $\mathrm{Spn}(\CE,\CM)$ is locally preordered. 

 \begin{proposition}\label{pb2cellinterchangeprop}
 Given $X\xra{m_*}Y\in \CM_*$ and $Z\xra{e^*}Y\in \CE^*$, there exists a diagram
 of the form
 \begin{eqnarray}\label{pb2cellinterchange}
\begin{aligned}
\xymatrix{
 J \ar[r]^-{\bar{e}^*} \ar[d]_-{\bar{m}_*}  & \ltwocell{ld}{ } X \ar[d]^-{m_*} \\
 Z \ar[r]_-{e^*}  & Y 
} 
\end{aligned}
\end{eqnarray}
in $\mathrm{Spn}(\CE,\CM)$, with $\bar{e}\in \CE$ and $\bar{m}\in \CM$, which is unique up to isomorphism.    
  \end{proposition}
  \begin{proof}
  Interpreting $m_*\circ \bar{e}^*\le e^*\circ \bar{m}_*$, we see that $\bar{m}\circ \bar{e}$ is forced to be an $(\CE,\CM)$ factorization of $e\circ m$.
  \end{proof}  
  
   \begin{proposition}
If $m\in \CM$ then $m_*$ is a discrete fibration in $\mathrm{Spn}(\CE,\CM)$;
that is, each functor $\mathrm{Spn}(\CE,\CM)(K,m_*)$ is a discrete fibration.  
  \end{proposition}

 \begin{proposition}\label{2wayFS}
 $(\CE^*,\CM_*)$ is a factorization system on the bicategory $\mathrm{Spn}(\CE,\CM)$.  
\end{proposition} 
\begin{proof}  
Every morphism $(d,R,m) : U \to W$ decomposes as $U\xra{d^*}R \xra{m_*}W$;
this decomposition $\CM_*\CE^*$ is unique up to isomorphism. 
The bipullback form of FS1 can be readily checked for this factorization.
\end{proof}  
  
 \begin{proposition}\label{pbpostar}
 Pullbacks in $\CC$ whose morphisms are all in $\CM$ are taken by $(-)_*$
 to bipullbacks in $\mathrm{Spn}(\CE,\CM)$.
 Also, pushouts in $\CC$ whose morphisms are all in $\CE$ are taken by $(-)^*$
 to bipullbacks in $\mathrm{Spn}(\CE,\CM)$.  
  \end{proposition}       

\section{Relations as spans of spans}

By regular categories we mean those in the sense of Barr \cite{Barr1971}
which admit all finite limits.
One characterization of the bicategory of relations in a regular category was
given in \cite{26}.
A relation from $X$ to $Z$ in a regular category is a jointly monomorphic span
from $X$ to $Z$; these are composed using span composition followed by factorization. Equivalently, a relation from $X$ to $Z$ is a subobject of $X\times Y$.

The category $\mathrm{Grp}$ of groups is regular. 
So relations are subgroups of products $X\times Z$. 
The Goursat Lemma \cite{Goursat} is a bijection between subgroups $S\le X\times Z$ of a cartesian product of groups $X$ and $Z$ and end-fixed isomorphism classes of diagrams
 \begin{equation}\label{zig-zag}
 \begin{aligned}
\xymatrix{X  & U \ar@{>->}[l]_-{m} \ar@{->>}[r]^-{d} &  Y & \ar@{->>}[l]_-{e} V \ar@{>->}[r]^-{n} & Z \ . 
}   
  \end{aligned}
\end{equation} 
To obtain $S$ from \eqref{zig-zag}, take the pullback $U\xla{\bar{e}}P\xra{\bar{d}}V$ of $U\xra{d}Y\xla{e}V$ then $S$ is the image of $P\xra{(m\bar{e},e\bar{d})}X\times Y$.
To obtain the zig-zag \eqref{zig-zag} from $S\hookrightarrow X\times Z$, factorize the two
restricted projections to obtain
\begin{equation*}
 \begin{aligned}
\xymatrix{X  & U \ar@{>->}[l]_-{m}  &  S \ar@{->>}[l]_-{e'} \ar@{->>}[r]^-{d'} &  V \ar@{>->}[r]^-{n} & Z \ , 
}   
  \end{aligned}
\end{equation*}   
then pushout $e'$ and $d'$ to obtain $d$ and $e$.

This motivates the definition of {\em relation} from $X$ to $Z$ in a category $\CC$
equipped with a suitable factorization system $(\CE,\CM)$ as an isomorphism
class of diagrams of the form \eqref{zig-zag}.
A good reference is \cite{Grandis} for the case where $\CC$ is Puppe-exact.

The starting point for the present paper was the simple observation that 
a relation diagram \eqref{zig-zag} is a span in $\mathrm{Spn}(\CE,\CM)$:
\begin{eqnarray}
\begin{aligned}
\xymatrix{(d,U,m,Y,e,V,n) \ : \ X & &  Y  \ar[ll]_-{(d,U,m)} \ar[rr]^-{(e,V,n)} & & Z \ .}
\end{aligned}
\end{eqnarray}

Write $\mathrm{Spn}[\CE,\CM]$ for the classifying category of the bicategory $\mathrm{Spn}(\CE,\CM)$; it has the same objects as $\mathrm{Spn}(\CE,\CM)$ and isomorphism classes $[e,S,m]$ of morphisms $(e,S,m)$. 
We would like to define the bicategory $\mathrm{Rel}(\CE,\CM)$ to be $\mathrm{Spn}(\mathrm{Spn}[\CE,\CM])$.
This is satisfactory as a definition of the 2-graph and vertical composition, but for the horizontal composition we need a way to compose spans in $\mathrm{Spn}[\CE,\CM]$.  

\section{A fake pullback construction}

Let $(\CE,\CM)$ be a suitable factorization system on a category $\CC$.
Although $\CC$ may not have all pullbacks, we will now show that 
$\mathrm{Spn}[\CE,\CM]$ does allow some kind of span composition
and this gives a composition of relations.
The construction and proof of associativity restructures that of \cite{Tsalenko1967}. 
We will see in Section~\ref{abstraction} that the properties of
$\mathrm{Spn}(\CE,\CM)$ established in Section~\ref{BicatEMSpn}
allow an abstract proof of associativity of composition of relations.

Take any cospan $U\xra{(d,R,m)}W\xla{(e,S,n)}V$ in $\mathrm{Spn}(\CE,\CM)$. 
Construct the diagram
 \begin{equation}\label{fpbinSpn(E,M)}
 \begin{aligned}
\xymatrix{
Q && Y \ar@{->>}[ll]_-{s} \ar@{>->}[rr]^-{j}  && V \\
X \ar@{->>}[u]^-{r} \ar@{>->}[d]_-{i} && Z \ar@{->>}[u]^-{\bar{e}} \ar@{>->}[d]^-{\bar{n}} \ar@{->>}[ll]^-{\bar{d}} \ar@{>->}[rr]_-{\bar{m}} && S \ar@{->>}[u]_-{e} \ar@{>->}[d]^-{n} \\
U && R   \ar@{->>}[ll]^-{d} \ar@{>->}[rr]_-{m}  && W 
}
 \end{aligned}
\end{equation} 
in which the bottom right square is a pullback of $R\mra W \ela S$, the bottom left
square is an $(\CE,\CM)$-factorization of the composite $Z\mra R \era U$, the top right square is an $(\CE,\CM)$-factorization of the composite $Z\mra S \era V$,
and the top left square is a pushout of the span $X\ela Z\era Y$.

We call the span $U\xla{(r,R,i)}Q\xra{(s,S,j)}V$ the {\em fake pullback} of the given cospan
$U\xra{(d,R,m)}W\xla{(e,S,n)}V$.
We obtain the diagram \eqref{fpbinSpn(E,M)star} in $\mathrm{Spn}(\CE,\CM)$.
The top left square comes from a pushout in $\CC$, the bottom right square
from a pullback in $\CC$, while the 2-cells come from factorizing an $\CM$ followed by an $\CE$ as an $\CE$ followed by an $\CM$.

\begin{equation}\label{fpbinSpn(E,M)star}
 \begin{aligned}
\xymatrix{
Q \ar[d]_-{r^*}  \ar[rr]^-{s^*} \pb &  & Y  \ar[d]|-{\bar{e}^*} \ar[rr]^-{j_*}  &&  \rtwocell{lld}{ } V \ar[d]^-{e^*} \\
X \ar[rr]|-{\bar{d}^*}  \ar[d]_-{i_*} && \ltwocell{lld}{ } Z  \ar[d]|-{\bar{n}_*} \pb  \ar[rr]|-{\bar{m}_*} && S  \ar[d]^-{n_*} \\
U \ar[rr]_-{d^*} && R   \ar[rr]_-{m_*}  && W 
}
 \end{aligned}
\end{equation} 
\begin{remark}\label{fakemonos}
\begin{itemize}
\item[a.] If $d$ is invertible, so is $s$. If $m$ is invertible, so is $j$.
\item[b.] If $(\CE,\CM)$ is proper (that is, every $\CE$ is an epimorphism and every 
$\CM$ is a monomorphism) then every morphism 
$\mathbf{r} : X \to Y$ of $\mathrm{Spn}(\CE,\CM)$ is a ``fake monomorphism''
in the sense that the fake pullback of $X\xra{\mathbf{r}}Y\xla{\mathbf{r}}X$
is the identity span $X\xla{1_X}X\xra{1_X}X$.  
\end{itemize}
\end{remark}
 
\section{An abstraction}\label{abstraction}

A bicategory $\CS$ is defined to be {\em fake pullback ready} when it is locally preordered and is equipped with a factorization system $(\CU,\CL)$ satisfying the following conditions:
\begin{itemize}
\item[V1.] bipullbacks of $\CU$s along $\CU$s exist and are in $\CU$, and bipullbacks of $\CL$s along $\CL$s exist and are in $\CL$;
\item[V2.] given $X\xra{a}Z\xla{x}Y$ with $a\in \CU$ and $x\in \CL$, there exists a
square  
\begin{eqnarray}\label{V2}
\begin{aligned}
\xymatrix{
 U \ar[r]^-{b} \ar[d]_-{y}  & \ltwocell{ld}{ } Y \ar[d]^-{x} \\
 X \ar[r]_-{a}  & Z \ ,
} 
\end{aligned}
\end{eqnarray}
with $b\in \CU$ and $y\in \CL$, which is unique up to equivalence;
\item[V3.] given a diagram
\begin{eqnarray}\label{V3.1}
\begin{aligned}
\xymatrix{
X\ar[r]^-{x} \ar[d]_-{r} \pb & Y \ar[r]^-{a} \ar[d]_-{s}  & \ltwocell{ld}{ } Z \ar[d]^-{t} \\
A \ar[r]_-{y} & B \ar[r]_-{b}  & C \ , 
} 
\end{aligned}
\end{eqnarray}
with the left square a bipullback, $r,s,t,x,y\in \CL$ and $a,b\in \CU$, and factorizations $a\circ x \cong v\circ c$ and
$b\circ y\cong w\circ d$ with $v, w\in \CL$ and $c, d\in \CU$, there exists
a diagram
\begin{eqnarray}\label{V3.2}
\begin{aligned}
\xymatrix{
X\ar[r]^-{c} \ar[d]_-{r} & \ltwocell{ld}{ } I \ar[r]^-{v} \ar[d]^-{q} \pb & Z \ar[d]^-{t} \\
A \ar[r]_-{d} & J \ar[r]_-{w}  & C 
} 
\end{aligned}
\end{eqnarray}
with the right square a bipullback and $q\in \CL$;
\item[V4.] given a diagram
\begin{eqnarray}\label{V4.1}
\begin{aligned}
\xymatrix{
D \ar[r]^-{u} \ar[d]_-{e} & E \rtwocell{ld}{ } \ar[r]^-{h} \ar[d]_-{f} \pb  &  F \ar[d]^-{g} \\
X \ar[r]_-{x} & Y \ar[r]_-{a}  & Z 
} 
\end{aligned}
\end{eqnarray}
with the right square a bipullback, $x,u\in \CL$ and $h, a,e,f,g\in \CU$, and factorizations $h\circ u\cong p\circ k$  and $a\circ x \cong v\circ c$
with $v, p\in \CL$ and $c, k\in \CU$, there exists
a diagram
\begin{eqnarray}\label{V4.2}
\begin{aligned}
\xymatrix{
D \ar[r]^-{k} \ar[d]_-{e} \pb &  K  \ar[r]^-{p} \ar[d]^-{j}  & F \rtwocell{ld}{ } \ar[d]^-{g} \\
X \ar[r]_-{c} & I \ar[r]_-{v}  & Z 
} 
\end{aligned}
\end{eqnarray}
with the left square a bipullback and $j\in \CU$.
\end{itemize}
\begin{proposition}\label{mainexample}
Let $(\CE,\CM)$ be a suitable factorization system on the category $\CC$.
The locally preordered bicategory $\mathrm{Spn}(\CE,\CM)$ is rendered 
fake pullback ready by the 
factorization system $(\CE^*,\CM_*)$ of Proposition~\ref{2wayFS}.
\end{proposition}
\begin{proof}
Condition V1 is provided by Proposition~\ref{pbpostar}.
Condition V2 is provided by Proposition~\ref{pb2cellinterchangeprop}.
Consider diagram~\eqref{V3.1} with $x_*,a^*,\dots$ replacing $x,a, \dots$ since $\CL = \CM_*$
and $\CU=\CE^*$ in this case.
The left square amounts to the pullback shown as the right-hand square on the left-hand side of \eqref{V3diag}.
The right-hand square with the 2-cell amounts to the factorization $s\circ a = b\circ t$.
Now form the pullback on the left of the left-hand side of \eqref{V3diag} and the pullback
on the right of the right-hand side of \eqref{V3diag}.
Since $btv=sav=sxc=yrc$, there exists a unique $q$ such that $dq=rc$ and $wq=tv$.
So we have the equal pastings as shown in \eqref{V3diag}.  
\begin{eqnarray}\label{V3diag}
\begin{aligned}
\xymatrix{
I \ar@{>->}[r]^-{c} \ar@{>->}[d]_-{v} \pb & X \ar@{->>}[r]^-{r} \ar@{>->}[d]_-{x} \pb & A \ar@{>->}[d]^-{y} \\
Z \ar@{>->}[r]_-{a} & Y \ar@{->>}[r]_-{s}  & B
} 
\qquad =
\qquad
\xymatrix{
I \ar@{->>}[r]^-{q} \ar@{>->}[d]_-{v} & J \ar@{>->}[r]^-{d} \ar@{>->}[d]^-{w} \pb & A \ar@{>->}[d]^-{y} \\
Z \ar@{->>}[r]_-{t} & C \ar@{>->}[r]_-{b}  & B 
} 
\end{aligned}
\end{eqnarray}
It follows that the left diagram on the right-hand side of \eqref{V3diag} is a pullback and, by SFS3, that $q\in \CE$. Diagram~\eqref{V3.2} results.

It is V4 which requires suitable factorization condition SFS5.
Consider diagram~\eqref{V4.1}.
We have the pushout on the right of the left-hand side of \eqref{V4diag} and
the factorization $fx=ue$. 
Form the pullback of $a$ and $x$ and note, using one direction of SFS5, that it gives the pushout on the left
of the left-hand side of \eqref{V4diag}.
Next, factorize $gv = pj$ through $K$ with $p\in \CM$ and $j\in \CE$.
Using functoriality of factorization FS1, we obtain a unique $k : K\to D$ with $kj=ec$ and $uk=hp$.
\begin{eqnarray}\label{V4diag}
\begin{aligned}
\xymatrix{
I \ar@{>->}[r]^-{v} \ar@{->>}[d]_-{c} \po & Z \ar@{->>}[r]^-{g} \ar@{->>}[d]_-{a} \po & F \ar@{->>}[d]^-{h} \\
X \ar@{>->}[r]_-{x} & Y \ar@{->>}[r]_-{f}  & E 
} 
\qquad =
\qquad
\xymatrix{
I \ar@{->>}[r]^-{j} \ar@{->>}[d]_-{c} & K \ar@{>->}[r]^-{p} \ar@{->>}[d]^-{k} & F \ar@{->>}[d]^-{h} \\
X \ar@{->>}[r]_-{e} & D \ar@{>->}[r]_-{u}  & E 
} 
\end{aligned}
\end{eqnarray}
It follows that both squares on the right-hand side of \eqref{V4diag} are pushouts.
Diagram~\eqref{V4.2} results using the other direction of SFS5 to see that the right
square on the right-hand side of \eqref{V4diag} is a pullback and hence $p_*k^*=h^*u_*$.
\end{proof}

Let $\CS$ be fake pullback ready. The {\em fake pullback} of a cospan
$U\xra{r}W\xla{s}V$ in $\CS$ is constructed as follows.
Factorize $r\cong x\circ a$ and $s\cong y\circ b$ with $a,b\in \CU$ and $x,y\in \CL$.
Using half of V1, take the bipullback of $x$ and $y$ as shown in the bottom right square of \eqref{fpbinabstracto}.
Now construct the bottom left and top right squares of \eqref{fpbinabstracto} using V2.
Using the other half of V1, we obtain the top left bipullback.  
\begin{equation}\label{fpbinabstracto}
 \begin{aligned}
\xymatrix{
Q \ar[d]_-{b'}  \ar[rr]^-{a'} \pb &  & Y  \ar[d]|-{\bar{b}} \ar[rr]^-{\bar{x}}  &&  \rtwocell{lld}{ } V \ar[d]^-{b} \\
X \ar[rr]|-{\bar{a}}  \ar[d]_-{\bar{y}} && \ltwocell{lld}{ } Z  \ar[d]|-{y'} \pb  \ar[rr]|-{x'} && S  \ar[d]^-{y} \\
U \ar[rr]_-{a} && R   \ar[rr]_-{x}  && W 
}
 \end{aligned}
\end{equation} 
The span $U\xla{\bar{y}b'}Q\xra{\bar{x}a'}V$ is our fake pullback of $U\xra{r}W\xla{s}V$.
\begin{proposition}
Fake pullbacks are symmetric. That is, if $U\xla{\bar{s}}Q\xra{\bar{r}}V$ is
a fake pullback of $U\xra{r}W\xla{s}V$ then $V\xla{\bar{r}}Q\xra{\bar{s}}U$ is
a fake pullback of $V\xra{s}W\xla{r}U$.   
\end{proposition}
\begin{proof}
In \eqref{fpbinabstracto}, the bipullbacks are symmetric and both 2-cells point to the boundary of the diagram. 
So the diagram is symmetric about its main diagonal.
\end{proof}
Note that, should a bipullback
\begin{eqnarray*}
\begin{aligned}
\xymatrix{
 U \ar[r]^-{b} \ar[d]_-{y}  & \ltwocell{ld}{\cong} Y \ar[d]^-{x} \\
 X \ar[r]_-{a}  & Z \ ,
} 
\end{aligned}
\end{eqnarray*}
of $X\xra{a}Z\xla{x}Y$ exist with $y\in \CL$ and $b\in \CU$, it would provide
the square for V2. This happens for example when $a$ is an identity, $b$ is
an identity, and $y=x$. Consequently:
\begin{proposition}
An identity morphism provides a fake pullback of an identity morphism along any morphism.    
\end{proposition}
\begin{eqnarray}\label{fakepbstack}
\begin{aligned}
\xymatrix{
P\ar[r]^-{\bar{t}} \ar[d]_-{\bar{\bar{s}}} \fpb & Q \ar[r]^-{\bar{r}} \ar[d]|-{\bar{s}} \fpb &  V \ar[d]^-{s} \\
X \ar[r]_-{t} & U \ar[r]_-{r}  & W  
} 
\qquad
\phantom{a}
\qquad
\xymatrix{
P\ar[r]^-{\bar{r}\bar{t}} \ar[d]_-{\bar{\bar{s}}} \fpb & V \ar[d]^-{s} \\
X \ar[r]_-{rt} & W 
} 
\end{aligned}
\end{eqnarray} 
\begin{proposition}
Fake pullbacks stack. That is, if the two squares on the left of \eqref{fakepbstack}
are fake pullbacks then so is the pasted square on the right of \eqref{fakepbstack}.  
\end{proposition}
\begin{proof}
Faced with a diagram like
\begin{equation*}
 \begin{aligned}
\xymatrix{
Q \ar[d]|-{u}  \ar[r]|-{u} \pb & Y  \ar[d]|-{u} \ar[r]|-{\ell}  &  \rtwocell{ld}{ } V \ar[d]|-{u}  \ar[r]|-{u} \pb & A  \ar[d]|-{u} \ar[r]|-{\ell}  &  \rtwocell{ld}{ } D \ar[d]|-{u} \\
X \ar[r]|-{u}  \ar[d]|-{\ell} & \ltwocell{ld}{ } Z  \ar[d]|-{\ell} \pb  \ar[r]|-{\ell} & S   \ar[r]|-{u}  \ar[d]|-{\ell} & \ltwocell{ld}{ } B  \ar[d]|-{\ell} \pb  \ar[r]|-{\ell} & E  \ar[d]|-{\ell} \\
U \ar[r]|-{u} & R   \ar[r]|-{\ell}  & W \ar[r]|-{u} & C   \ar[r]|-{\ell}  & F 
}
 \end{aligned}
\end{equation*} 
in which the arrows marked $u$ are in $\CU$ and those marked $\ell$ are in $\CL$, 
we apply condition V3 to the middle bottom two squares and condition V4 to the
middle top two squares to obtain
\begin{equation*}
 \begin{aligned}
\xymatrix{
Q \ar[d]|-{u}  \ar[r]|-{u} \pb & Y  \ar[d]|-{u} \ar[r]|-{u} \pb  &   I \ar[d]|-{u}  \ar[r]|-{\ell}  & A \rtwocell{ld}{ } \ar[d]|-{u} \ar[r]|-{\ell}  &  \rtwocell{ld}{ } D \ar[d]|-{u} \\
X \ar[r]|-{u}  \ar[d]|-{\ell} & \ltwocell{ld}{ } Z  \ar[d]|-{\ell}   \ar[r]|-{u} & \pb P \ltwocell{ld}{ }  \ar[r]|-{\ell}  \ar[d]|-{\ell} &  B  \ar[d]|-{\ell} \pb  \ar[r]|-{\ell} & E  \ar[d]|-{\ell} \\
U \ar[r]|-{u} & R   \ar[r]|-{u}  & P' \ar[r]|-{\ell} & C   \ar[r]|-{\ell}  & F 
}
 \end{aligned}
\end{equation*}
which is again a fake pullback.
\end{proof}

As a corollary of all this we have:
\begin{theorem}
Let $\CS$ be a fake pullback ready bicategory. 
There is a category $\mathrm{Spn}[\CS]$ whose objects are those of $\CS$,
whose morphisms are isomorphism classes of spans in $\CS$, and
whose composition is defined by fake pullback.
\end{theorem}  

\begin{remark}
Given Remark~\ref{fakemonos}, we might call $\CS$ {\em proper} when the identity span provides a fake pullback of each morphism with itself. In this case, each morphism $X\xra{\mathbf{r}}Y$ in 
$\mathrm{Spn}[\CS]$ satisfies $\mathbf{r}\mathbf{r}^{\circ}\mathbf{r}=\mathbf{r}$
where $\mathbf{r}^{\circ} : Y\to X$ is the reverse span of $\mathbf{r}$.   
\end{remark}  

\begin{center}
--------------------------------------------------------
\end{center}

\appendix

\end{document}